\documentclass[a4paper,10pt]{article}
\usepackage[utf8]{inputenc}
\usepackage{pdflscape}
\usepackage[linesnumbered,ruled,vlined]{algorithm2e}
\usepackage{comment,xcomment,amssymb,amsmath,color}
\usepackage{rotating,xypic,array}
\usepackage{latexsym,mathtools,amsthm}
\usepackage{subfig,rotating}
\usepackage{tikz,fancyvrb}
\usepackage{booktabs}  
\usetikzlibrary{arrows.meta,automata}
\usepackage{enumitem}
\usepackage{url} 
\usepackage{longtable}
\usepackage[section]{placeins}
\usepackage[longtable]{multirow}
\usepackage[colorlinks,citecolor=blue]{hyperref}

\title{A construction of Einstein solvmanifolds not based on nilsolitons}
\author{Diego Conti, Federico A. Rossi and Romeo Segnan Dalmasso}

\newcommand\CC{C\nolinebreak\hspace{-.05em}\raisebox{.4ex}{\relsize{-3}{\textbf{+}}}\nolinebreak\hspace{-.10em}\raisebox{.4ex}{\relsize{-3}{\textbf{+}}}}
\newtheorem{theorem}{Theorem}[section]
\newtheorem{lemma}[theorem]{Lemma}

\newtheorem{proposition}[theorem]{Proposition}
\theoremstyle{definition}

\newtheorem{example}[theorem]{Example}

\theoremstyle{remark}
\newtheorem{remark}[theorem]{Remark}

\newcommand{\abs}[1]{\left\vert#1\right\vert}
\newcommand{\R}{\mathbb{R}}
\newcommand{\lie}[1]{\mathfrak{#1}}     
\newcommand{\g}{\lie{g}}
\newcommand{\Z}{\mathbb{Z}}

\newcommand{\C}{\mathbb{C}}

\newcommand{\hook}{\lrcorner\,}

\newcommand{\SO}{\mathrm{SO}}

\newcommand{\id}{\operatorname{Id}}   

\newcommand{\Span}[1]{\operatorname{Span}\left\{#1\right\}}

\newcommand{\tran}[1]{\hspace{.2mm}\prescript{t\hspace{-.5mm}}{}{#1}}

\newcommand{\st}{\;|\;}          
\DeclareMathOperator{\Ric}{Ric} 

\DeclareMathOperator{\Aut}{Aut}

\DeclareMathOperator{\diag}{diag}

\DeclareMathOperator{\Der}{Der}
\DeclareMathOperator{\ad}{ad}

\DeclareMathOperator{\logsign}{logsign}

\DeclareMathOperator{\Tr}{tr}
\DeclareMathOperator{\sign}{sign}

\newcolumntype{C}{>{$}c<{$}}
\newcolumntype{L}{>{$}l<{$}}
\newcolumntype{R}{>{$}r<{$}}

\usepackage[backend=biber]{biblatex}
\newcommand{\diagonal}{\mathrm{D}}

\begin{document}
\VerbatimFootnotes
\maketitle
\begin{abstract}
We construct indefinite Einstein solvmanifolds that are standard, but not of pseudo-Iwasawa type. Thus, the underlying Lie algebras take the form $\g\rtimes_D\R$, where $\g$ is a nilpotent Lie algebra and $D$ is a nonsymmetric derivation. Considering nonsymmetric derivations has the consequence that $\g$ is not a nilsoliton, but satisfies a more general condition.

Our construction is based on the notion of nondiagonal triple on a nice diagram. We present an algorithm to classify nondiagonal triples and the associated Einstein metrics. With the use of a computer, we obtain all solutions up to dimension $5$, and all solutions in dimension $\leq9$ that satisfy an additional technical restriction.

By comparing curvatures, we show that the Einstein solvmanifolds of dimension $\leq 5$ that we obtain by our construction are not isometric to a standard extension of a nilsoliton.
\end{abstract}

\renewcommand{\thefootnote}{\fnsymbol{footnote}}
\footnotetext{\emph{MSC class 2020}: 53C25 (\emph{Primary}); 53C30, 53C50, 22E25 (\emph{Secondary})}
\renewcommand{\thefootnote}{\arabic{footnote}}

\section*{Introduction}
Einstein manifolds of negative curvature and maximal symmetry have been studied for decades. After contributions by many authors (see \cite{LL14structure,JP17step,BL18eucledean,Jab15Bstrongly}), it was proved in \cite{Bohm_Lafuente_2023} that every homogeneous Einstein Riemannian manifold of negative curvature can be represented as a solvmanifold, i.e. a solvable Lie group endowed with a left-invariant metric (a statement previously known as the Alekseevsky conjecture).

Given an Einstein solvmanifold, at the Lie algebra level, there is an orthogonal decomposition $\g\rtimes\lie a$, where $\lie g$ is the nilradical and $\lie a$ an abelian subalgebra (\cite{Lauret_2010}); such a decomposition is called a standard decomposition. Up to isometry, one can then assume that $\lie a$ acts by symmetric derivations (\cite{Heber_1998}); the standard decomposition is then said to be of Iwasawa type. Furthermore, the restriction of the metric to $\lie g$ satisfies the equation
\begin{equation}
\label{eqn:nilsoliton}
\Ric=\lambda \id+D, \quad D\in\Der\g;
\end{equation}
one then says that $\g$ is a \emph{nilsoliton}; the terminology is motivated by the fact that a left-invariant metric on a nilpotent Lie group satisfies~\eqref{eqn:nilsoliton} if and only if it is a Ricci soliton~\cite{Lauret_2001,Jablonski2014}.

Indefinite homogeneous Einstein manifolds are less constrained (partly because restrictions on the sign of the scalar curvature cease to be significant if one does not fix the signature): they need not be solvmanifolds, or even diffeomorphic to $\R^n$ (consider the symmetric spaces $\SO_0(p,q)/\SO(p)\times\SO(q)$), and even if one restricts to solvmanifolds, the Einstein condition does not imply the existence of a standard decomposition $\g\rtimes\lie a$. Furthermore, if a standard decomposition does exist, it may not be the case that $\lie a$ acts by symmetric derivations, even up to isometry.

Nevertheless, the constructive aspects of the positive-definite theory generalize to arbitrary signature. For instance, large classes of Einstein solvmanifolds can be obtained by extending indefinite nilsolitons (see~\cite{Conti_Rossi_best_before,Conti_Rossi_2022}). The solvmanifolds obtained this way admit standard decompositions of pseudo-Iwasawa type, i.e. they take the form $\g\rtimes_D\R$ where $D$ is symmetric and $\g$ is the nilradical.

Constructing Einstein solvmanifolds which are not of pseudo-Iwasawa type is more difficult. In~\cite{Conti_Rossi_Segnan_2023} we obtained the first such examples in the guise of standard Sasaki-Einstein solvmanifolds; since Sasaki solvmanifolds can never be of pseudo-Iwasawa type, those metrics are not isometric to any standard solvmanifold of pseudo-Iwasawa type (\cite[Proposition~2.6]{Conti_Rossi_Segnan_korean}).

The construction of~\cite{Conti_Rossi_Segnan_2023} is based on a generalization of the nilsoliton condition: on a nilpotent Lie algebra, one considers a metric and a derivation $D$ with symmetric part $D^s=\frac12(D+D^*)$ such that for $\tau=\pm1$ the following conditions involving the Ricci operator hold:
\begin{equation}
\label{eqn:generalizednilsolitoningeneral}
\Ric=\tau\biggl(- \Tr ((D^s)^2) \id- \frac12[D,D^*] +(\Tr D)D^s\biggr), \quad \Tr (\ad v\circ D^*)=0,\ v\in\g.
\end{equation}
Rather than attack this equation directly, the method of~\cite{Conti_Rossi_Segnan_2023} is to find solutions by inverting contact (symplectic) reduction, which is a peculiar feature of Sasaki (K\"ahler) geometry.

By contrast, in this paper we leave contact geometry aside and give a direct construction of solutions of~\eqref{eqn:generalizednilsolitoningeneral}. Unlike the Sasaki case, we do not have a general criterion to exclude that the resulting metrics are isometric to solvmanifolds of pseudo-Iwasawa type, but we show that this is not generally the case by explicit curvature computations in low dimensions.

Our construction uses Lie algebras admitting a special type of basis, introduced in \cite{LauretWill:EinsteinSolvmanifolds,Nikolayevsky} under the name of \emph{nice bases}. The metrics we consider have an orthonormal nice basis, but the derivation $D$ is not diagonal in this basis. It is a feature of our construction that $D$ is always diagonalizable, although this does not follow by any means from~\eqref{eqn:generalizednilsolitoningeneral}.

The precise ansatz we impose on $D$ is that for every $i=1,\dotsc, n$, the derivation $D$ has at most one nonzero element on either the $i$-th row or the $i$-th column which is not on the diagonal. Some of the entries of $D$ are forced to be zero by the condition $\Tr (\ad V\circ D^*)=0$. Having assumed that the nice basis is orthogonal, the left-hand side of~\eqref{eqn:generalizednilsolitoningeneral} is diagonal relative to that basis, so we need to impose that the right-hand side is zero off the diagonal; the resulting constraints on $D$ are linear, due to the special nature of our ansatz. The diagonal part of the right-hand side of~\eqref{eqn:generalizednilsolitoningeneral}, however, depends nonlinearly on $D$. A full characterization of the conditions is given in Lemma~\ref{lemma:almostdiagonal}.

In order to obtain a linear problem, we change the point of view; rather than fix the nice Lie algebra and consider an arbitrary diagonal metric, we consider the set of nice Lie algebras which share the same set of indices ${i,j,k}$ such that $[e_i,e_j]=c_{ijk}e_k\neq0$ (in the terminology of~\cite{Conti_Rossi_2019}, they have the same nice diagram), and leave the structure constants relative to a fixed orthonormal basis as unknowns. We then determine the derivation $D$ in terms of the nice diagram, the set $\mathcal A$ of indices corresponding to offdiagonal nonzero entries of $D$,
and a function $A\colon\mathcal A\to\R$  characterizing the actual entries of $D$ in a suitable sense. We call $(D,\mathcal A,A)$ a \emph{nondiagonal triple}, and determining it is a linear problem. For a fixed nondiagonal triple, the structure constants $c_{ijk}$ must then be computed in such a way that the Jacobi identity holds, the Ricci operator takes the required form, and $D$ is a derivation. This is a nonlinear problem, but it can be solved effectively using the fact that finding a diagonal metric on a nice Lie algebra with prescribed Ricci operator is a  linear problem in the squared structure constants $c_{ijk}^2$. The full conditions that must be satisfied in order to obtain a solution of~\eqref{eqn:generalizednilsolitoningeneral} from a nondiagonal triple are given in Theorem~\ref{thm:almostdiagonalwithdiagrams}.

In view of a classification, we then introduce a suitable notion of equivalence between nondiagonal triples, taking into account sign flipping in the basis elements and automorphisms of the nice diagrams.

We then present Algorithm~\ref{alg:classification}, which classifies nondiagonal triples and the associated solutions of~\eqref{eqn:generalizednilsolitoningeneral} up to equivalence; our implementation of the algorithm can be found at~\cite{esticax}. This algorithm is mostly effective in low dimensions or under two technical assumptions, namely that the so-called root matrix is surjective and that the linear system determining $A$ has a unique solution, ensuring that the equations to be solved do not depend on parameters.

We obtain a classification of nondiagonal triples and the associated solutions of~\eqref{eqn:generalizednilsolitoningeneral}, up to equivalence, both in dimension $\leq 5$ (Tables~\ref{table:34} and~\ref{table:5}) and in dimension $\leq 9$ under the two technical assumptions outlined above (see ancillary files). Each entry in these tables determines a standard Einstein solvmanifold in one dimension higher, which is not of pseudo-Iwasawa type. We argue that these metrics differ from the known metrics obtained by extending a nilsoliton by computing the curvature (Proposition~\ref{prop:notthesame}).

\textbf{Acknowledgments}
The authors acknowledge GNSAGA of INdAM and
the PRIN project n. 2022MWPMAB ``Interactions between Geometric Structures and Function Theories''.

D. Conti acknowledges the MIUR Excellence Department Project awarded to the Department of Mathematics, University of Pisa, CUP I57G22000700001.

F.A.~Rossi acknowledges the INdAM-GNSAGA project CUP E55F22000270001 ``Curve algebriche e loro applicazioni''.

\section{Generalized nilsolitons and diagonal metrics}
In this section we recall the results, terminology and notation from \cite{LauretWill:EinsteinSolvmanifolds,Nikolayevsky,Conti_Rossi_2019,Conti_Rossi_2020} that will be used in the sequel.

We shall consider metrics on a Lie algebra $\g$, i.e. nondegenerate scalar products that determine a left-invariant pseudo-Riemannian metric on a Lie group with Lie algebra $\g$; the Levi-Civita connection and its curvature can then be expressed at the Lie algebra level. In particular, we shall denote by $\Ric\colon\g\to\g$ the Ricci operator; the Einstein condition reads $\Ric=\lambda \id$, where $\id$ is the identity operator on $\g$; the notation $\id_\g$ will also be used when necessary.

Given a Lie algebra $\tilde\g$ with a metric $\tilde g$, we say that a \emph{standard decomposition} is an orthogonal decomposition $\tilde \g=\g\rtimes\lie a$, where $\g$ is a nilpotent ideal and $\lie a$ an abelian subalgebra. As $\g$ and $\lie a$ are required to be orthogonal, the restriction of the metric to $\g$ is nondegenerate, and will be denoted by $g$.

The standard decomposition is generally not unique. A standard decomposition is said to be of \emph{pseudo-Iwasawa type} if $\ad X$ is symmetric for all $X$ in $\lie a$. Given an Einstein solvmanifold of pseudo-Iwasawa type $\tilde \g=\g\rtimes\lie a$, $\g$ satisfies the nilsoliton equation~\eqref{eqn:nilsoliton}. We are interested in Einstein metrics which admit a standard decomposition, but not one of pseudo-Iwasawa type. The standard decomposition will take the form $\g\rtimes_D\R$, where $D$ differs from its symmetric part $D^s=\frac12(D+D^*)$. The condition that the metric on the nilpotent factor $\g$ must satisfy is given by the following result, generalizing~\eqref{eqn:nilsoliton} to the case where $D$ is not assumed to be
symmetric:
\begin{theorem}[\protect{\cite[Proposition~2.1]{Conti_Rossi_Segnan_2023}}]
\label{thm:extendgeneralizednilsoliton}
Let $\g$ be a nilpotent Lie algebra with a pseudo-Riemannian metric $g$, $D$ a derivation and $\tau=\pm1$. Then the metric $\tilde g=g+\tau e^0\otimes e^0$ on
$\tilde\g=\g\rtimes_D\Span{e_0}$ is Einstein if and only if
\[\Ric=\tau\biggl(- \Tr ((D^s)^2) \id- \frac12[D,D^*] +(\Tr D)D^s\biggr), \qquad \Tr (\ad v\circ D^*)=0,\ v\in\g;\]
in this case, $\widetilde{\Ric}= - \tau\Tr ((D^s)^2)\id_{\tilde\g}$.
\end{theorem}
For the remainder of the article we set $\tau=1$, thus solving the equation
\begin{equation}\label{eqn:generalizednilsoliton}
\Ric=- \Tr ((D^s)^2) \id- \frac12[D,D^*] +(\Tr D)D^s,\qquad \Tr (\ad v\circ D^*)=0,\ v\in\g.
\end{equation}
One can recover the metrics with $\tau=-1$ by flipping the overall sign of the metric.

Our approach to finding solutions of~\eqref{eqn:generalizednilsoliton} is through nice Lie algebras. Given a Lie algebra $\g$, a basis $e_1,\dotsc, e_n$ with dual basis $e^1,\dotsc, e^n$ is a \emph{nice basis} if each  $[e_i,e_j]$, $e_i\hook de^j$ is a multiple of a basis element (see \cite{LauretWill:EinsteinSolvmanifolds,Nikolayevsky}); a nice Lie algebra is a Lie algebra endowed with a nice basis. To such a basis one can associate a directed graph with nodes $\{1,\dotsc, n\}$, and such that $i\to k$ is an arrow if and only if $[e_i,e_j]$ is a nonzero multiple of $e_k$ for some $j$. The arrow $i\to k$ is then decorated with the label $j$; we will write $i\xrightarrow{j}k$. It is clear from the definition that for fixed $i,j$ there can be at most one arrow $i\xrightarrow{\bullet} j$ and at most one arrow $i\xrightarrow{j}\bullet$. In addition, if $i\xrightarrow{j}k$ is an arrow, then $j\xrightarrow{i}k$ is also an arrow. Nilpotency implies that the graph is acyclic; additionally, it satisfies a condition involving concatenated arrows which follows from the Jacobi identity. All these conditions (see~\cite{Conti_Rossi_2019} for details) define a class of labeled directed acyclic graphs known as \emph{nice diagrams}. These are combinatorial objects which can in principle be classified in any fixed dimension.

Having assigned a nice diagram $\Delta$, any nice Lie algebra whose underlying nice diagram is $\Delta$ can be expressed by assigning structure constants $c_{ijk}$ to every arrow $i\xrightarrow{j}k$ in such a way that the Jacobi identity holds. However, the solution may not exist or be unique; thus, the correspondence between nice Lie algebra and nice diagrams is not bijective. In particular, if each basis vector $e_i$ is rescaled by a constant $g_i$, a new solution $\{c'_{ijk}\}$ is obtained by $c'_{ijk}=\frac{g_k}{g_ig_j}c_{ijk}$. This can be seen as follows.

Given a nice diagram $\Delta$, one defines the \emph{root matrix} $M_\Delta$, which has a row of the form
\[\bigl(0,\dotsc, \underbrace{-1}_i,\dotsc, \underbrace{-1}_j,\dotsc, \underbrace{1}_k,\dotsc, 0\bigr)\]
for every pair of arrows $i\xrightarrow{j}k$, $j\xrightarrow{i}k$. The rows of $M_\Delta$ represent the weights for the action of the diagonal group $D_n$
(in the basis $\{e^i\otimes e_i\}$ of its Lie algebra $\lie d_n$) on the $m$-dimensional space of structure constants $\{c_{ijk}\}$. Accordingly, we can view $M_\Delta$ as a linear map $\lie d_n\to\lie d_m$, which exponentiates to a map $D_n\to D_m$ which we denote by $e^{M_\Delta}$; identifying $\lie d_n$ with $\R^n$ and $\lie d_m$ with $\R^m$, if the $I$-th row of $M_\Delta$ corresponds to $i\xrightarrow{j}k$, the $I$-th component of $e^{M_\Delta}(g_1,\dotsc, g_n)$ is $\frac{g_k}{g_ig_j}$.

It then follows that when $M_\Delta$ is surjective the structure constants $c_{ijk}$ can be normalized to $\pm1$ by rescaling. Otherwise, continuous families of Lie algebras with the same diagram may occur. In order to avoid the difficulty of having to solve equations depending on parameters, we will only consider nice diagrams  with surjective root matrix in this paper.

Nice Lie algebras have two features which make them a good candidate for the construction of solutions of~\eqref{eqn:nilsoliton} and~\eqref{eqn:generalizednilsoliton}. First, one has fine control over the derivations, as the following holds:
\begin{proposition}[\protect{\cite[Proof of Theorem 3]{Nikolayevsky}}]
\label{prop:diagonalpartofderivation}
Let $\g$ be a Lie algebra with a nice basis $e_1,\dotsc, e_n$. Then every derivation of $\g$ splits as the sum of a diagonal derivation $\sum\lambda_ie^i\otimes e_i$ and a derivation with zeroes on the diagonal, $\sum_{i\neq j} a_{ij}e^i\otimes e_j$.
\end{proposition}
Moreover, diagonal derivations can be computed solely in terms of the nice diagrams. Given a vector $v\in\R^n$, we will denote by $v^\diagonal$ the diagonal $n\times n$ matrix with entries determined by $v$. On a nice Lie algebra $\g$, identified with $\R^n$ by fixing a nice basis, one then has:
\begin{equation}
\label{eqn:diagonalderivation}
v^\diagonal\in\Der\g\iff v\in\ker M_\Delta.
\end{equation}

The second and most important feature of nice Lie algebras is that any diagonal metric has diagonal Ricci operator. The problem of determining  a diagonal metric with prescribed Ricci tensor is then expressed by $n$ equations in $n$ unknowns; furthermore, it can be split into a linear and a polynomial problem. We will use the following formula:
\begin{proposition}[\protect{\cite[Theorem 2.3]{Conti_Rossi_2020}}]
\label{prop:Ric_M_Delta}
If $\g$ is a nice Lie algebra with nice diagram $\Delta$ and structure constants $\{c_I\}$, then the Ricci operator of the diagonal metric $g^\diagonal=g_1e^1\otimes e^1+\dotsc + g_ne^n\otimes e^n$ is
\begin{equation}
\label{eqn:Ric_M_Delta}
\Ric=\frac12(\tran M_\Delta X)^\diagonal, \qquad \left(\frac{x_{I}}{c_{I}^2}\right)=e^{M_\Delta}(g).
\end{equation}
\end{proposition}
Notice that the particular value of $g$ satisfying
$\left(\frac{x_{I}}{c_{I}^2}\right)=e^{M_\Delta}(g)$ is only relevant to establish the signature: if $g$ and $h$ are solutions of \eqref{eqn:Ric_M_Delta} that differ by a positive factor in each entry, say $g_i=t_ih_i$, then the map that rescales each $e_i$ by $\sqrt{t_i}$ is a Lie algebra isomorphism, i.e.  the metrics $g^\diagonal$ and $h^\diagonal$  effectively correspond to the same metric and Lie algebra written relative to different bases. To account for the signature, we will introduce the notation
\[\logsign x=\begin{cases}
1 & x<0 \\
0 & x>0
\end{cases},\qquad  \logsign X=(\logsign x_I)_I.\]
Equation~\eqref{eqn:Ric_M_Delta} shows that if $I$ represents the arrow $i\xrightarrow{j}k$, $x_{I}$ has the same sign as $g_k/(g_ig_j)$; in term of the $\operatorname{mod} 2$ reduction of the root matrix, denoted $M_{\Delta,2}$, we can write
\begin{equation}
\label{eqn:signature}
\logsign X=M_{\Delta,2}\logsign g.
\end{equation}

Finally, let us recall that an automorphism of the nice diagram is a permutation  $\sigma$ of its nodes that maps arrows to arrows, i.e. $\sigma_i\xrightarrow{\sigma_j}\sigma_k$ is an arrow whenever $i\xrightarrow jk$ is an arrow. It is natural to consider nice diagrams up to automorphisms, which corresponds to considering nice Lie algebras up to reordering of the basis. An explicit computation of automorphisms of a nice diagram is given below in Example~\ref{example:of_the_algorithm}.

\section{Constructing generalized nilsolitons}
\label{sec:constructing}
The aim of the section is to develop tools to construct solutions to the ``generalized nilsoliton'' equation~\eqref{eqn:generalizednilsoliton} on a nice Lie algebra $\g$. We consider the case $\Tr D\neq0$, which gives rise to nonunimodular solvmanifolds $\g\rtimes_D\R$. The  derivation $D$ will turn out to be diagonalizable, but not diagonal relative to the nice basis.

In the following construction, we consider derivations $D$ which are \emph{almost} diagonal, meaning that their nondiagonal entries are indexed by a set
\begin{equation}
\label{eqn:A}
\begin{split}
\mathcal{A}&=\{(i_1,j_1),\dotsc, (i_k,j_k)\},\\
& i_1,\dotsc, i_k,j_1,\dotsc, j_k \text{ pairwise distinct elements in } {1,\dotsc, n}.
\end{split}
\end{equation}

The next lemma will be the guide to determine what properties should be imposed on the set $\mathcal{A}$ and derivation $D$.
Without loss of generality, we will require $\Tr D=\Tr D^2$. Notice that this normalization affects the Einstein constant $-\Tr ((D^s)^2)$.
\begin{lemma}
\label{lemma:almostdiagonal}
Let $\g$ be a nice Lie algebra with a diagonal metric
\[g=g_1e^1\otimes e^1+\dotsc + g_ne^n\otimes e^n,\]
let $\mathcal{A}$ be as in \eqref{eqn:A}, and let $D$ be a derivation of the form
\[D=\lambda_1 e^1\otimes e_1 + \dotsc + \lambda_n e^n\otimes e_n + \sum_{(i,j)\in \mathcal{A}}a_{ji}e^i\otimes e_{j},\]
where the $a_{ji}$ are nonzero. Assume furthermore that $\Tr D=\Tr D^2\neq0$. Then~\eqref{eqn:generalizednilsoliton} holds if and only if for $(i,j)\in \mathcal A$
\[
e_i\hook de^j=0, \qquad \Tr D = \lambda_i-\lambda_j,\qquad
a_{ji}^2= 2 \frac{g_i}{g_j}  A_j^i\Tr D,
\]
where $A_j^i$ are constants such that
\begin{equation}
\label{eqn:generalizednik}
\Tr (D\mu^\diagonal)=\biggl(1+ \sum_{(i,j)\in \mathcal A}A_j^i\biggr)\Tr \mu^\diagonal+ \sum_{(i,j)\in \mathcal A}A_j^i (\mu_j-\mu_i), \quad \mu^\diagonal\in\Der\g,
\end{equation}
and
\begin{multline*}
\frac1{\Tr D}\Ric=-\biggl(1 + \sum_{(i,j)\in \mathcal A}A_j^i\biggr)\id
-\sum_{(i,j)\in \mathcal A}A_j^i(e^j\otimes e_j-e^i\otimes e_i)\\
+(\lambda_1 e^1\otimes e_1 + \dotsc + \lambda_n e^n\otimes e_n ).
\end{multline*}
\end{lemma}
\begin{proof}
We compute
\[D^*=\lambda_1 e^1\otimes e_1 + \dotsc + \lambda_n e^n\otimes e_n + \sum_{(i,j)\in \mathcal A}\frac{g_j}{g_i}a_{ji}e^j\otimes e_i.\]
Since $\g$ is nice, the diagonal part of $D$ is a derivation (see Proposition~\ref{prop:diagonalpartofderivation}). Therefore, $\Tr \ad v\circ D^*=0$ if and only if
\begin{multline*}
0=\Tr \biggl(\ad v\circ \sum_{(i,j)\in \mathcal A}\frac{g_j}{g_i}a_{ji}e^j\otimes e_i)\biggr)= \sum_{(i,j)\in A}\frac{g_j}{g_i}a_{ji}e^j([v,e_i])\\
= -v\hook \sum_{(i,j)\in \mathcal A}\frac{g_j}{g_i}a_{ji}(e_i\hook de^j), \quad v\in\g.
\end{multline*}
By the nice condition, the terms $e_i\hook de^j=0$ can only be linearly dependent if they are zero.

We also compute
\begin{gather*}
[D,D^*]=\sum_{(i,j)\in \mathcal A} \biggl(a_{ji}(\lambda_i-\lambda_j)\bigl(\frac{g_j}{g_i}e^j\otimes e_i+e^i\otimes e_j)+ \frac{g_j}{g_i}a_{ji}^2(e^j\otimes e_j-e^i\otimes e_i)\biggr),\\
D^s=\lambda_1 e^1\otimes e_1 + \dotsc + \lambda_n e^n\otimes e_n + \frac{1}{2} \sum_{(i,j)\in \mathcal A}a_{ji}(e^i\otimes e_j+ \frac{g_j}{g_i}e^j\otimes e_i),\\
\Tr ((D^s)^2)=\lambda_1^2+\dotsc + \lambda_n^2 + \frac12\sum_{(i,j)\in \mathcal A} a_{ji}^2\frac{g_j}{g_i}=\Tr (D^2)+ \frac12\sum_{(i,j)\in \mathcal A} a_{ji}^2\frac{g_j}{g_i}.
\end{gather*}
Since we assume $a_{ji}\neq0$, the offdiagonal part of~\eqref{eqn:generalizednilsoliton} is satisfied if and only if for $(i,j)\in \mathcal A$
\[-\frac12(\lambda_i-\lambda_j)\bigl(\frac{g_j}{g_i}e^j\otimes e_i+e^i\otimes e_j)+(\Tr D)( \frac{1}{2}e^i\otimes e_j+ \frac{g_j}{2g_i}e^j\otimes e_i)=0,\]
i.e.
\[\Tr D = \lambda_i-\lambda_j, \qquad (i,j)\in \mathcal A.\]
On the other hand, the diagonal part of~\eqref{eqn:generalizednilsoliton} gives
\begin{multline*}
\Ric=-\biggl(\sum_i\lambda_i^2 + \frac12\sum_{(i,j)\in\mathcal A} a_{ji}^2\frac{g_j}{g_i}\biggr)\id
-\sum_{(i,j)\in\mathcal A} \frac{g_j}{2g_i}a_{ji}^2(e^j\otimes e_j-e^i\otimes e_i)\\
+(\Tr D)\biggl(\sum_i\lambda_i e^i\otimes e_i\biggr).
\end{multline*}
Setting $A_j^i=\frac{g_j a_{ji}^2}{2g_i\Tr D}$ and dividing by $\Tr D$, we can write
\[\frac1{\Tr D}\Ric=-\biggl(\frac{\Tr D^2}{\Tr D} + \sum_{(i,j)\in\mathcal A} A_j^i\biggr)\id
-\sum_{(i,j)\in\mathcal A} A_j^i(e^j\otimes e_j-e^i\otimes e_i)+\sum_i\lambda_i e^i\otimes e_i.\]
If $\mu^\diagonal$ is a diagonal derivation, since $\Tr(\Ric\circ \mu^\diagonal)=0$, we compute
\[0=-\biggl(1+\sum_{(i,j)\in \mathcal{A}} A_j^i\biggr)\Tr \mu^\diagonal-\sum_{(i,j)\in \mathcal{A}} A_j^i(\mu_j-\mu_i)+\Tr D\mu^\diagonal.\qedhere\]
\end{proof}
\begin{remark}
A derivation $D$ satisfying the conditions of Lemma~\ref{lemma:almostdiagonal} must necessarily be diagonalizable. Indeed, suppose $D$ is as in Lemma~\ref{lemma:almostdiagonal}. Up to reordering indices, we can assume that $D$ is upper triangular; in particular, the diagonal elements $\lambda_i$ are the eigenvalues.
If $\lambda=\lambda_j$ is an eigenvalue, then the $j$-th row of $D-\lambda I$ consists of zeroes except possibly for the $i$-th entry, if $(i,j)\in\mathcal A$; in the latter case, however, $\lambda_i-\lambda_j=\Tr D\neq 0$, so the $i$-th row of $D-\lambda I$ is zero everywhere except at the $i$-th entry. This shows that the rank of $D-\lambda I$ coincides with the number of diagonal elements distinct from $\lambda$. Thus, $D$ is diagonalizable.
\end{remark}

\begin{remark}
A similar construction as in Lemma~\ref{lemma:almostdiagonal} could be in principle considered for $\Tr D=0$; in that case, $\g\rtimes_D\R$ would be unimodular, and the derivation not diagonalizable. We do not know whether this will produce new examples; we plan to study this in future work.
\end{remark}

\begin{example}
Consider the Lie algebra $\g=(0,0,e^{12},e^{13})$; this notation, inspired by~\cite{Salamon:ComplexStructures}, means that there is a fixed basis $e_1,\dotsc, e_4$ such that the dual basis $e^1,\dotsc, e^4$ satisfies
\[de^1=0=de^2, \qquad de^3=e^1\wedge e^2, \qquad de^4=e^1\wedge e^3.\]
The generic diagonal derivation is
\[(-a+b,2a-b,a,b)^\diagonal.\]
We consider $\mathcal A=\{(1,2)\}$, i.e. the derivation
\[D=(-a+b,2a-b,a,b)^\diagonal+a_{21}e^1\otimes e_2.\]
The condition $\Tr D=2a+b=\lambda_1-\lambda_2$ gives
\[2a+b=-a+b-(2a-b);\]
together with the equations~\eqref{eqn:generalizednik},
this gives a linear system in $a,b,A_2^1$ with solution
\[a= \frac{7}{51},\qquad A_2^1= -\frac{11}{17},\qquad b = \frac{35}{51},\]
i.e.
\[D=\left(\begin{array}{cccc}\frac{28}{51}&0&0&0\\a_{21}&-\frac{7}{17}&0&0\\0&0&\frac{7}{51}&0\\0&0&0&\frac{35}{51}\end{array}\right).\]
In addition, we have
\[{a_{21}^{2}=-\frac{1078}{867} \frac{g_1}{g_2}}.\]
By~\eqref{eqn:Ric_M_Delta}, we must solve
\[\frac1{2\Tr D}\tran M_\Delta(X)=\left(-\frac{23}{51},-\frac{2}{17},-\frac{11}{51},\frac{1}{3}\right),\]
which gives
\[X=\left(\frac{196}{867},\frac{98}{153}\right).\]
Now we must solve $e^{M_\Delta}(g)=X$, i.e.
\[\frac{g_3}{g_1g_2}=\frac{196}{867}, \qquad \frac{g_4}{g_1g_3}=\frac{98}{153}.\]
A particular solution is given by
\[g_1=3,\quad g_2=-22,\quad g_3=-\frac{4312}{289},\quad g_4=-\frac{422576}{14739},\quad a_{21}=\frac7{17}.\]
Therefore we see that the metric
\[3e^1\otimes e^1-22e^2\otimes e^2-\frac{4312}{289}e^3\otimes e^3-\frac{422576}{14739}e^4\otimes e^4+e^5\otimes e^5\]
on the Lie algebra $\tilde{\g}=\g\rtimes_D\Span{e_5}$
\[\left(\frac{28}{51}e^{15},-\frac7{17}e^{25}+\frac7{17}e^{15},\frac7{51}e^{35}+e^{12},\frac{35}{51}e^{45}+e^{13},0\right)\]
is Einstein with
\[\Ric=-\frac{98}{289}\id.\]
\end{example}

In Lemma~\ref{lemma:almostdiagonal}, the structure constants are fixed, the parameters $\{a_{ji}\}$ must satisfy linear conditions that make $D$ a derivation, and the $A_j^i$ linear conditions that follow from $\Tr(\Ric\circ D)=0$, but their relation to the metric is nonlinear. We now take a differerent point of view: we do not fix the structure constants, but only the nice diagram. The $A_j^i$ and the diagonal part of the derivations are determined linearly, and then we impose conditions on the structure constants so that the Ricci operator takes the required form and $D$ is a derivation. Notice that since we allow the structure constants to vary, we may assume that the nice basis is orthonormal, i.e. the $g_i$ equal $\pm1$. This leads to the following definition.

Let $\Delta$ be a nice diagram with $n$ nodes. We will identify nodes with numbers $\{1,\dotsc, n\}$. Let  $\mathcal A$ be a nonempty subset of $N(\Delta)\times N(\Delta)$. Let $D\colon\R^n\to\R^n$ be a linear map, $A\colon \mathcal A\to\R^*$ a function; we will write $A_j^i$ for $A(i,j)$.
We say that  $(D,A,\mathcal{A})$ is a \emph{nondiagonal triple} if the following conditions hold:
\begin{enumerate}
[label=(N\arabic*)]
\item\label{cond:nondiagonal1}
whenever $(i,j)$ is in $\mathcal A$, $i\neq j$;
\item\label{cond:nondiagonal2}
whenever $(i,j),(i',j')$ are distinct elements of $\mathcal A$, then the elements $i,j,i',j'$ are pairwise distinct;
\item\label{cond:nondiagonal3} if $(i,j)$ is in $\mathcal A$, then there is no arrow $i\xrightarrow{k}j$;
\item\label{cond:nondiagonal4}
denoting by $(\lambda_1,\dotsc, \lambda_n)$ the diagonal elements of $D$, for every
$\mu\in\ker M_\Delta$,
\[\sum_i \lambda_i\mu_i=\biggl(1+ \sum_{(i,j)\in \mathcal A}A_j^i\biggr)\sum_i \mu_i+ \sum_{(i,j)\in \mathcal A}A_j^i  (\mu_j-\mu_i);\]
\item\label{cond:nondiagonal5}
$\lambda_1+\dotsc + \lambda_n=\lambda_i-\lambda_j$ whenever $(i,j)\in \mathcal A$;
\item\label{cond:nondiagonal6}
$D$ takes the form
\[D=\lambda_1 e^1\otimes e_1 + \dotsc + \lambda_n e^n\otimes e_n + \sum_{(i,j)\in \mathcal A}a_{ji}e^i\otimes e_{j},\]
where $a_{ji}^2=\abs{2A_j^i\Tr D}\neq0$ and $(\lambda_1,\dotsc,\lambda_n)$ is in $\ker M_\Delta$.
\end{enumerate}

\begin{remark}
\label{remark:A_ij_determined}
Conditions~\ref{cond:nondiagonal4} and~\ref{cond:nondiagonal5} always determine the $\lambda_i$ completely, but not necessarily the $A_j^i$. Indeed, let $f_{ij}(\lambda)=\lambda_1+\dots + \lambda_n-(\lambda_i-\lambda_j)$ and let
\[V=\{\lambda\in\ker M_\Delta \st f_{ij}(\lambda)=0\},\qquad\dim V=n_5\]
be the space of diagonal derivations satisfying~\ref{cond:nondiagonal5}.
Then, for $\lambda$ in $V$, imposing~\ref{cond:nondiagonal4} for any $\mu$ in $V$ yields $n_5$ independent equations
\[\sum_i \lambda_i\mu_i=\sum_i \mu_i,\]
thus completely determining $\lambda\in\ker M_\Delta$.

On the other hand, consider $W\subset\ker M_\Delta$ such that $V\oplus W=\ker M_\Delta$. Imposing~\ref{cond:nondiagonal4} for $\mu$ in $W$ gives an equation of the form
\begin{equation}
\label{eqn:N4onW}
\sum_{i,j}f_{ij}(\mu)A_j^i=\Tr(\lambda \mu) -\Tr(\mu).
\end{equation}
The matrix of this linear system in the unknowns $A_j^i$ has columns $f_{ij}(\mu)$, where each generator $\mu$ of $W$ determines a row.

The rows are independent because $W$ intersects $V$ trivially. Hence, the system always admits a solution, which is unique precisely when $\dim W=\abs{\mathcal A}$, i.e.
\[n_5=\dim\ker M_\Delta-\abs{\mathcal A}.\]
This holds if and only if condition~\ref{cond:nondiagonal5} imposes exactly $\abs{\mathcal A}$ linearly independent equations.
\end{remark}

\begin{remark}
Any Lie algebra admits a semisimple derivation $N$ satisfying $\Tr Nf=\Tr f$ for every derivation $f$, unique up to automorphisms, known as a \emph{Nikolayevsky derivation}, or \emph{pre-Einstein derivation}; it is known that for Riemannian solutions of~\eqref{eqn:nilsoliton} one must have $D=N$ up to multiples and automorphisms (see~\cite[Theorem 1]{Nikolayevsky}). If $\mathcal{A}$ is empty, condition~\ref{cond:nondiagonal4} implies that $D$ is the Nikolayevsky derivation. In general, however, the Nikolayevsky derivation will not satisfy~\ref{cond:nondiagonal5} (for instance, when its eigenvalues are positive, as is the case for Riemannian nilsolitons). In addition, $D$ may only equal the Nikolayevsky derivation if the linear equations~\eqref{eqn:N4onW} are homogeneous. A nontrivial solution in $A$ only exists if the columns $f_{ij}(\mu)$ are linearly dependent, i.e. $\abs{\mathcal{A}}>\dim W$, which means that the equations of condition~\ref{cond:nondiagonal5} are linearly dependent. In this paper we will focus on Lie algebras of dimension $\leq 5$, for which the Nikolayevsky derivation has positive eigenvalues, and those of higher dimension for which the equations of condition~\ref{cond:nondiagonal5} are independent. Therefore, none of the metrics we construct have $D=N$ up to a multiple.
\end{remark}

In the next theorem, we restate the construction of Lemma~\ref{lemma:almostdiagonal} using nice diagrams and nondiagonal triples. We will use the notation $[x]$ to represent the vector all of whose entries equal $x$ in $\R^n$, where $n$ is to be deduced from the context.
\begin{theorem}
\label{thm:almostdiagonalwithdiagrams}
Let $\Delta$ be a nice diagram. Let $(\mathcal{A},D,A)$ be a nondiagonal triple. Let $X$ be a vector such that
\begin{equation}
\label{eqn:generalizednilsolitonX}
\frac1{2\Tr D}(\tran M_\Delta X) = \biggl[-1-\sum_{(i,j)\in \mathcal{A}} A_j^i\biggr]+\sum_{(i,j)\in \mathcal{A}} A_j^i(e_i-e_j)+(\lambda_1,\dotsc, \lambda_n).
\end{equation}
Suppose $\epsilon\in\{\pm1\}^n$ satisfies
\[M_{\Delta,2}(\logsign \epsilon)=\logsign X,\qquad \epsilon_i/\epsilon_j = \sign (A_j^i\Tr D), \quad (i,j)\in \mathcal A.\]
Suppose $\g$ is a nice Lie algebra with diagram $\Delta$ such that the structure constants satisfy $c_I^2=\abs{x_I}$ and $D$ is a derivation. Then the diagonal metric $\epsilon^\diagonal$
satisfies~\eqref{eqn:generalizednilsoliton}.
\end{theorem}
\begin{proof}
The hypotheses of Lemma~\ref{lemma:almostdiagonal} are satisfied. Thus, the Ricci operator is diagonal and~\eqref{eqn:generalizednilsoliton} is equivalent to
\[\frac1{\Tr D}\Ric = \biggl[-1-\sum_{(i,j)\in \mathcal{A}} A_j^i\biggr]^\diagonal+\sum_{(i,j)\in \mathcal{A}} A_j^i(e_i-e_j)^\diagonal+(\lambda_1,\dotsc, \lambda_n)^\diagonal.\]
By Proposition~\ref{prop:Ric_M_Delta}, we must solve
\[e^{M_\Delta}(\epsilon)=\left(\frac{x_I}{c^2_I}\right)=\bigl(\sign x_I\bigr).\]
Taking log signs, this boils down to
\[M_{\Delta,2}(\logsign \epsilon)=\logsign X.\qedhere\]
\end{proof}
Every metric obtained with Theorem~\ref{thm:almostdiagonalwithdiagrams} determines a standard Einstein solvmanifold which is not of pseudo-Iwasawa type by applying Theorem~\ref{thm:extendgeneralizednilsoliton}. We will illustrate this concretely in one example.
\begin{example}
Consider the diagram with four nodes and arrows $1\xrightarrow{2}3, 1\xrightarrow{3}4$. Then we have a nondiagonal triple given by
\begin{gather*}
\mathcal{A}=\{(1,2)\}, \qquad
D=\left(\begin{array}{cccc}\frac{28}{51}&0&0&0\\\pm \frac{7}{51}  \sqrt{3} \sqrt{22}&-\frac{7}{17}&0&0\\0&0&\frac{7}{51}&0\\0&0&0&\frac{35}{51}\end{array}\right), \qquad A_2^1=-\frac{11}{17}.
\end{gather*}
A solution of~\eqref{eqn:generalizednilsolitonX} is given by
\[X=\left(\frac{196}{867},\frac{98}{153}\right),\]
giving rise to the Lie algebra
\[\left(0,0,\frac{14}{867}  \sqrt{867} e^{12}, \sqrt{\frac{98}{153}} e^{13}\right).\]
It is easy to check that $D$ is always a derivation. We have two choices of $\epsilon$ that satisfy the conditions of Theorem~\ref{thm:almostdiagonalwithdiagrams}, namely
\[\epsilon =(-1,1,-1,1),\qquad  \epsilon=(1,-1,-1,-1).\]
The resulting $5$-dimensional solvable Lie algebra is
\[\left(\frac{28}{51} e^{15}+\frac{7}{51} \sqrt{3} \sqrt{22} e^{25},-\frac{7}{17}  e^{25},\frac{7}{51} e^{35}+\frac{14}{867} \sqrt{867} e^{12},\frac{35}{51} e^{45}+ \sqrt{\frac{98}{153}} e^{13},0\right).\]
\end{example}

It will be convenient to give the following definition. Given a nondiagonal triple $(\mathcal A, A,D)$, a \emph{nondiagonal solution} is a pair $(\{c_I\},\epsilon)$ such that the conditions of Theorem~\ref{thm:almostdiagonalwithdiagrams} hold for some $X$.
Notice that  $X$  is uniquely determined by the data.

A nondiagonal solution determines an Einstein solvmanifold applying Theorem~\ref{thm:almostdiagonalwithdiagrams} and Theorem~\ref{thm:extendgeneralizednilsoliton}. We conclude this section by discussing when two Einstein solvmanifolds obtained in this way should be regarded as equivalent.

In general, identifying whether two solvmanifolds are isometric as pseudo-Riemannian manifolds is a difficult problem. There is a straightforward sufficient condition that one can test, as explained in \cite[Theorem~5.6]{Azencott_Wilson_1976}, \cite[Proposition~1.1]{Conti_Rossi_Segnan_2023}. The observation is that if $D$ is replaced with a different derivation $D'$ that commutes with $D$ and such that $D-D'$ is skew-symmetric relative to the metric, the resulting extensions $\g\rtimes_{D'}\R$ and $\g\rtimes_D \R$ lead to  isometric pseudo-Riemannian manifolds. However, no two metrics obtained from Theorem~\ref{thm:almostdiagonalwithdiagrams} can be related in this way. Indeed, fix two nondiagonal triples $(\mathcal A,A,D)$,  $(\mathcal A',A',D')$. The form of the metric implies that the space of
skew-symmetric endomorphisms is spanned by
\[e^i\otimes e_j-\epsilon_i\epsilon_j e^j\otimes e_i, \quad i\neq j.\]
The only possibility in order to have a pair $(j,i)$ such that both $(i,j)$ and $(j,i)$ are nonzero entries of $D-D'$ is if $(i,j)\in\mathcal A$ and $(j,i)\in\mathcal A'$ or viceversa. On the other hand, $e^i\otimes e_j-\epsilon_i\epsilon_j e^j\otimes e_i$ will not commute with $D$ in this case, because $\lambda_i-\lambda_j=\Tr D\neq0$.

A finer notion of equivalence we can consider is identifying two extensions $\g\rtimes_D\R$ and $\g'\rtimes_{D'}\R$ if they are related by a Lie algebra isomorphism which is also an isometry. We observe that since in the construction $\Tr D\neq0$, $\g$ is the nilradical of $\g\rtimes_D\R$. Therefore, an isomorphism $\g\rtimes_D\R\to\g'\rtimes_{D'}\R$  would induce an isomorphism of the nilradicals $\g$ and $\g'$. Since the nilradicals are nice Lie algebras, we consider two nondiagonal solutions equivalent if denoting by $\g$ the nice Lie algebra determined by $\Delta$, $\{c_I\}$ and $g$ the metric defined by $\epsilon$, with $\g',g'$ defined similarly, there is an  equivalence of nice Lie algebras $\g\to\g'$ that maps $g$ to $g'$ and $D$ to $D'$.
\begin{lemma}
\label{lemma:equivalent}
Given nice diagrams $\Delta$, $\Delta'$, nondiagonal triples $(\mathcal A, A,D)$, $(\mathcal A', A',\allowbreak D')$, and nondiagonal solutions $(\{c_I\},\epsilon)$, $(\{c'_I\},\epsilon')$, the nondiagonal solutions are equivalent if and only if there is an isomorphism of nice diagrams $f\colon\Delta\to\Delta'$ and $\delta\in\{\pm1\}^n$ such that
\[\lambda_{f(i)}'=\lambda_{i}, \qquad a_{f(i)f(j)}'=\delta_i\delta_j a_{ij}, \qquad
c'_{f(i)f(j)f(k)}=\delta_i\delta_j\delta_kc_{ijk}, \qquad \epsilon'_{f(i)}=\epsilon_i.\]
\end{lemma}
\begin{proof}
Since the nice bases are assumed to be orthonormal, the isomorphism is essentially obtained by a permutation of the indices $\{1,\dotsc, n\}$, which corresponds to an isomorphism of the diagrams, preceded by sign flips.
\end{proof}
If we fix $\Delta$, $(\delta,f)$ as in Lemma~\ref{lemma:equivalent} is an element of the group $\Z_2^n\rtimes\Aut(\Delta)$, acting on the set of nondiagonal triples. Equivalence of nondiagonal triples amounts to being in the same orbit for this action.

\section{Algorithm, implementation and results}
The discussion of Section~\ref{sec:constructing} leads naturally to an algorithm to classify nondiagonal solutions up to equivalence, which we give explicitly in Algorithm~\ref{alg:classification}. The algorithm reflects the construction in a straightforward way; the only subtlety is that for efficiency the quotient under the action of $\Z_2^n\rtimes\Aut(\Delta)$ is taken in two steps: in the outer iteration through the possible index sets $\mathcal{A}$, only one index set is taken in each orbit for the natural action of $\Aut(\Delta)$, and at the end of the iteration, the resulting nondiagonal triples are factored by the action of $\Z_2^n\rtimes(\Aut \Delta)^{\mathcal A}$, where $(\Aut \Delta)^{\mathcal A}$  indicates the stabilizer of ${\mathcal A}$.
\begin{algorithm}
\SetKwInOut{Input}{input}
\SetKwInOut{Output}{output}
\caption{\label{alg:classification}Classification of nondiagonal solutions up to equivalence}
\Input{The dimension $n$}
\Output{Nondiagonal solutions up to equivalence on nice Lie algebras of dimension $n$}
\For{$\Delta$ nice diagram with $n$ nodes (one in each isomorphism class)}{
\For{$\mathcal{A}\subset N(\Delta)\times N(\Delta)$ nonempty and satisfying \ref{cond:nondiagonal1}--\ref{cond:nondiagonal3} (one in each $\Aut(\Delta)$-orbit)}{
$L \leftarrow \emptyset$\\
Compute $(\lambda_1,\dotsc, \lambda_n)\in\ker M_\Delta$ and $A\colon\mathcal A\to\R^*$ such that~\ref{cond:nondiagonal4} and~\ref{cond:nondiagonal5} hold.\\
Compute $X$ so that~\eqref{eqn:generalizednilsolitonX} holds.\\
\For{$\{c_I\}$ with $c_I^2=\abs{x_I}$ and $D$  as in~\ref{cond:nondiagonal6} such that the Jacobi identity holds and $D$ is a derivation}{
\For{$\epsilon$ such that Theorem~\ref{thm:almostdiagonalwithdiagrams} holds}
{add $(\mathcal A,A,D,\{c_I\},\epsilon)$ to $L$}
}
Add a section of $L/(\Z_2^n\rtimes\Aut(\Delta)^\mathcal A)$ to output
}
}
\end{algorithm}

\begin{example}
\label{example:of_the_algorithm}
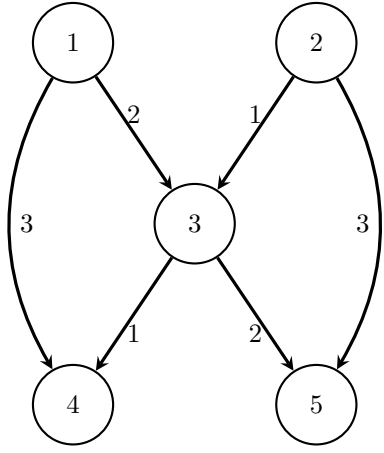
\begin{figure}[ht]
{\centering
\begin{tikzpicture}[scale=0.8]
\begin{scope}[every node/.style={circle,thick,draw}, minimum size=3em]
    \node (1) at (-2,3) {$1$};
    \node (2) at (2,3) {$2$};
    \node (3) at (0,0) {$3$};
    \node (4) at (-2,-3) {$4$};
    \node (5) at (2,-3) {$5$};
\end{scope}
\begin{scope}[>={stealth[black]}, every edge/.style={draw,very thick}]
    \path [->] (1) edge node [above]{$2$} (3);
    \path [->] (2) edge node [above]{$1$} (3);
    \path [->] (1) edge [bend right=30, looseness=1] node [right]{$3$} (4);
    \path [->] (2) edge [bend left =30, looseness=1] node [left]{$3$} (5);
    \path [->] (3) edge node [below] {$1$} (4);
    \path [->] (3) edge node [below]{$2$} (5);
\end{scope}
\end{tikzpicture}
\caption{\label{Fig.Diagram1}Diagram of $(0,0,e^{12},e^{13},e^{23})$}
}
\end{figure}
Consider the nice diagram $\Delta$ with $5$ nodes and arrows $1\xrightarrow{2}3$, $1\xrightarrow{2}4$, $1\xrightarrow{3}5$ (see Figure~\ref{Fig.Diagram1}). The only nontrivial automorphism of $\Delta$ is the involution $(1\,2)(4\,5)$. The sets $\mathcal{A}$ that satisfy \ref{cond:nondiagonal1}--\ref{cond:nondiagonal3} are reduced using the automorphism of $\Delta$, so $(1,2)$ is in the same orbit as $(2,1)$; similarly, $(5,3)$ is in the same orbit as $(4,3)$. Since $D=d_{ij}e_i\otimes e^j$ in the nondiagonal triple $(\mathcal A, A,D)$ is a derivation of a Lie algebra $\g$ with diagram $\Delta$, and since the generic derivation of the Lie algebra $(0,0,c_{123} e^{12},c_{134} e^{13}, c_{235} e^{23})$ is given by
\[
\begin{pmatrix}
d_{33}-d_{22} & d_{12} & 0 & 0 & 0 \\
d_{21} & d_{22} & 0 & 0 & 0 \\
d_{31} & d_{32} & d_{33} & 0 & 0 \\
d_{41} & d_{42} & \frac{c_{134}}{c_{123}} d_{32} & 2 d_{33}-d_{22} & \frac{c_{134}}{c_{235}} d_{12} \\
d_{51} & d_{52} & -\frac{c_{235}}{c_{123}} d_{31} & \frac{c_{235}}{c_{134}} d_{21} & d_{22}+d_{33} \\
\end{pmatrix},\]
the elements of $D$ in position $(2,1)$ and $(5,4)$ are simultaneously zero or nonzero, and similarly for: $(1,2)$ and $(4,5)$; $(3,2)$ and $(4,3)$; $(3,1)$ and $(5,3)$. In addition, $d_{42}$ and $d_{51}$ are allowed to be nonzero; the remaining nondiagonal elements vanish.
Taking all into account, the following are the allowed $\mathcal{A}$ up to the action of $\Aut(\Delta)$:
\[
\mathcal{A}_1=\{(1,5)\},\qquad\mathcal{A}_2=\{(1,5),(2,4)\},\qquad\mathcal{A}_3=\{(1,2),(4,5)\}.
\]
We see that
\[M_{\Delta}=
\begin{pmatrix}
-1 & -1 & 1 & 0 & 0 \\
-1 & 0 & -1 & 1 & 0 \\
0 & -1 & -1 & 0 & 1 \\
\end{pmatrix},\]
so $\ker M_{\Delta}=\{(\lambda_3-\lambda_2,\lambda_2,\lambda_3,2\lambda_3-\lambda_2,\lambda_2+\lambda_3) \st \lambda_2,\lambda_3\in\R\}$.
Conditions \ref{cond:nondiagonal4}--\ref{cond:nondiagonal5} give the following constraints for each case:
\begin{description}[style=unboxed,leftmargin=0cm]
\item[{Case $\mathcal{A}_1=\{(1,5)\}.$}]
\[
\begin{cases*}
2 \lambda_2+5 A_5^1+5=7 \lambda_3\\
2 \lambda_2=\lambda_3+A_5^1\\
5 \lambda_3=-2 \lambda_2
\end{cases*},
\]
with solution $A_5^1=-\frac{5}{7}$, $\lambda_3=\frac{5}{42}$ and $\lambda_2=-\frac{25}{84}$.
Equation~\eqref{eqn:generalizednilsolitonX} gives
\[X={\left(\frac{25}{63},\frac{25}{84},\frac{25}{84},\right)};\]
solving $c_{ijk}^2=\abs{x_{ijk}}$ we get
\[c_{123}=\frac{5}{3 \sqrt{7}},\quad c_{134}=\frac{5}{2 \sqrt{21}},\quad c_{235}=\frac{5}{2 \sqrt{21}};\]
and by $a^2_{ij}=\abs{2A_j^i\Tr D}$ we obtain
\[a_{51}=\pm \frac{5}{7} \sqrt{\frac{5}{3}}.\]
However, $M_{\Delta,2}(\logsign \epsilon)=\logsign X$ gives $\epsilon_1\epsilon_2\epsilon_3=+1=\epsilon_2\epsilon_3\epsilon_5$, thus
\[\epsilon_1=\epsilon_1\epsilon_2\epsilon_3\epsilon_5=\epsilon_5,\]
but $\epsilon_i/\epsilon_j = \sign (A_j^i\Tr D)$ gives $\epsilon_1\epsilon_5=-1$, which is impossible.

\item[{Case $\mathcal{A}_2=\{(1,5),(2,4)\}.$}]
\[
\begin{cases*}
2 \lambda_2+7 A_4^2+5 A_5^1+5=7 \lambda_3\\
2 \lambda_2+A_4^2=\lambda_3+A_5^1\\
5 \lambda_3=-2 \lambda_2\\
2 \lambda_2=7 \lambda_3
\end{cases*},
\]
with solution $A_5^1=-\frac{5}{12}$, $A_4^2=-\frac{5}{12}$, $\lambda_3=0$ and $\lambda_2=0$. This does not give a nondiagonal triple because $\Tr D$ is zero.

\item[{Case $\mathcal{A}_3=\{(1,2),(4,5)\}.$}]
\[
\begin{cases*}
2 \lambda_2+4 A_2^1+4 A_5^4+5=7 \lambda_3\\
2 \lambda_3+2A_2^1+2A_5^4=4\lambda_2\\
\lambda_2+2 \lambda_3=0
\end{cases*},
\]
with solution $A_2^1=-A_5^4-\frac{25}{31}$, $\lambda_3=\frac{5}{31}$ and $\lambda_2=-\frac{10}{31}$.
Equation~\eqref{eqn:generalizednilsolitonX} gives
\[X={\left(\frac{100}{961},x_{134},\frac{150}{961}-x_{134}\right)},\quad A_2^1=-\frac{31}{50} x_{134}-\frac{11}{31},\quad A_5^4=\frac{31}{50} x_{134}-\frac{14}{31};\]
solving $c_{ijk}^2=\abs{x_{ijk}}$ we get
\[c_{123}=\frac{10}{31},\qquad c_{134}=\sqrt{\abs{x_{134}}},\qquad c_{235}=\sqrt{\abs{\frac{150}{961}-x_{134}}};\]
and by $a^2_{ij}=\abs{2A_j^i\Tr D}$ we obtain
\[a_{21}=\pm 5 \sqrt{\frac{2}{31} \left| A_2^1\right|},\qquad a_{54}=\pm5 \sqrt{\frac{2}{31}\left| -A_2^1-\frac{25}{31}\right|}.\]
Since $D$ is a derivation, we obtain $a_{54}c_{134}=c_{234}a_{21}$. Taking the square and substituting the previous equation we get the following
\[\abs{ x_{134}}  \abs{ \frac{31 }{50}x_{134}-\frac{14}{31}} =\abs{\frac{150}{961}-x_{134}} \abs{-\frac{31}{50}x_{134}-\frac{11}{31}},\]
which gives three solutions for $x_{134}$, namely $\frac{75}{961}, \frac{25}{961} \left(3-5 \sqrt{3}\right)$ and $\frac{25}{961} \left(5 \sqrt{3}+3\right)$. However, $M_{\Delta,2}(\logsign \epsilon)=\logsign X$ and $\epsilon_i/\epsilon_j = \sign (A_j^i\Tr D)$ have solutions only for $x_{134}=\frac{75}{961}$, hence
we get
\[x_{134}=\frac{75}{961},\quad A_2^1=-\frac{25}{62},\quad A_5^4=-\frac{25}{62},\quad a_{21}= \pm\frac{25}{31},\quad a_{54}=\pm\frac{25}{31},\]
and the metrics are
\begin{equation}
\label{eqn:metricsonexample}
\begin{gathered}
-e^1\otimes e^1+e^2\otimes e^2-e^3\otimes e^3+e^4\otimes e^4-e^5\otimes e^5,\\  e^1\otimes e^1-e^2\otimes e^2-e^3\otimes e^3-e^4\otimes e^4+e^5\otimes e^5.
\end{gathered}
\end{equation}
Finally, we observe that changing the sign of $e_1,e_3$ and $e_5$ amounts to switching the sign of $a_{12}$ and $a_{45}$, and by Lemma~\ref{lemma:equivalent} we only need to consider the case $a_{21}>0$.
\end{description}
So we conclude that up to equivalence the only solution is given by the Lie algebra $(0,0,\frac{10}{31}e^{12},\allowbreak \frac{5 \sqrt{3}}{31}e^{13},\frac{5 \sqrt{3}}{31}e^{23})$ with metrics~\eqref{eqn:metricsonexample} and derivation \[D=\left(\frac{15}{31},-\frac{10}{31},\frac{5}{31},\frac{20}{31},-\frac{5}{31}\right)^\diagonal+\frac{25}{31}e^{1}\otimes e_2+\frac{25}{31}e^{4}\otimes e_5.\]
\end{example}

Applying Algorithm~\ref{alg:classification} poses several problems. We illustrate the issues and how we addressed them in our implementation~\cite{esticax}, based on the \CC{} library GiNaC~\cite{Bauer_Frink_Kreckel_2002}.
\begin{enumerate}
\item At line $2$ of Algorithm~\ref{alg:classification}, a classification of nice diagrams up to automorphisms is needed. An algorithm to this effect was introduced in~\cite{Conti_Rossi_2019} and implemented in~\cite{demonblast}; thus, we resorted to the same code.

\item For a fixed nice diagram, the set of possible $\mathcal{A}$ is generally quite large; however, as observed in Example~\ref{example:of_the_algorithm}, at line $2$ it is not necessary to consider all possible $\mathcal A$, but only those such that for some nice Lie algebra with diagram $\Delta$ there exist derivations whose nondiagonal entries are exactly parameterized by $\mathcal A$. In general, the nice diagram does not determine the nice Lie algebra uniquely; however, it is always possible to write down a linear space that contains the space of derivations of all Lie algebras with a given nice diagram. This optimization also has the effect of eliminating nice diagrams which are not associated to any nice Lie algebra.

\item At lines $4$--$5$, computing $\lambda_1,\dotsc, \lambda_n$, $A$ and $X$ are linear computations. These may result in solutions depending on parameters: as observed in Remark~\ref{remark:A_ij_determined}, if~\ref{cond:nondiagonal5} does not determine $\lambda_1,\dotsc, \lambda_n$, then the $A_j^i$ are not uniquely determined. Additionally, it may be the case that~\eqref{eqn:generalizednilsolitonX} does not determine $X$ if $\tran M_\Delta$ is not injective, i.e. the root matrix $M_\Delta$ is not surjective.

\item At line $6$, nonlinear computations take place: $X$ determines the structure constants, but square roots appear in the expression. Simple equations such as those of Example~\ref{example:of_the_algorithm}, case $\mathcal{A}_3$ can be solved automatically by rationalizing and solving a second degree equation in one variable, and we implemented this in~\cite{esticax}, but this becomes hopeless as free parameters increase or when equations contain the sum of three square roots.
\item At line $9$, we need to extract a section. For this, we used the explicit form of the group action given in Lemma~\ref{lemma:equivalent} and a simple scheme where the set $L$ is progressively reduced by an iteration that eliminates elements that are in the orbit of preceding elements.
\end{enumerate}
For the reasons outlined above, in dimension $6$ and higher, our software is not able to solve all cases.  With this in mind, we have restricted our classifications to
$n\leq 5$, and $6\leq n\leq9$ with surjective root matrix and $\mathcal{A}$ chosen so that~\ref{cond:nondiagonal5} consists of $\abs{\mathcal{A}}$ independent equations. Notice that in dimension $5$ the root matrix is automatically surjective. The resulting solutions of~\eqref{eqn:generalizednilsoliton}, each giving rise to an Einstein solvmanifold in one dimension higher, are given in Tables~\ref{table:34}, \ref{table:5}, A, B, C and D (see the ancillary files). Each table row contains a Lie algebra $\lie g$, a derivation $D$, and then the list of compatible metrics. The derivation $D$ is expressed as a sum $v+\sum a_{ij}e^i\otimes e_j$, where $v$ is a vector representing the diagonal derivation $v^\diagonal$. Since the nice basis is orthonormal, the metric is specified by giving the indices of the timelike vectors in the basis; thus, for instance, $12$ represents the metric $\diag(-1,-1,1,\dotsc, 1)$. The set of admissible signatures is denoted by $\mathbf{S}$. We obtain:
\begin{theorem}
Every solution of~\eqref{eqn:generalizednilsoliton} arising from a nondiagonal triple on a nice diagram with $n\leq 5$ is equivalent to exactly one entry in Tables~\ref{table:34} or~\ref{table:5}.
\end{theorem}

\begin{theorem}
Every solution of~\eqref{eqn:generalizednilsoliton} arising from a nondiagonal triple on a nice diagram with $6\leq n\leq 9$, a surjective root matrix and $\mathcal{A}$ chosen so that~\ref{cond:nondiagonal5} consists of $\abs{\mathcal{A}}$ independent equations is equivalent to exactly one entry in Tables A, B, C, D (see ancillary files).
\end{theorem}

{\setlength{\tabcolsep}{2pt}
\begin{longtable}[htpc]{>{\ttfamily}c C >{\renewcommand{\arraystretch}{1.2}}C C}
\caption{\label{table:34}Solutions of~\eqref{eqn:generalizednilsoliton} obtained with Algorithm~\ref{alg:classification} with $n=3,4$}\\
\toprule
\textnormal{Name $\Delta$} & \g &   D & \mathbf{S} \\
\midrule
\endfirsthead
\multicolumn{4}{c}{\tablename\ \thetable\ -- \textit{Continued from previous page}} \\
\toprule
\textnormal{Name $\Delta$} & \g &   D & \mathbf{S} \\
\midrule
\endhead
\bottomrule\\[-7pt]
\multicolumn{4}{c}{\tablename\ \thetable\ -- \textit{Continued to next page}} \\
\endfoot
\bottomrule\\[-7pt]
\endlastfoot
31:1&0,0,\frac{4}{7} e^{12}&\begin{array}{cc}
(\frac{6}{7},-\frac{2}{7},\frac{4}{7})\\
+\frac{8}{7} e^1\otimes e_2
\end{array}
&\{13, 23\}\\
\midrule
3:1&0,0,0&\begin{array}{cc}
(1,-\frac{1}{5},\frac{2}{5})\\
+\frac{6}{5} e^1\otimes e_2
\end{array}
&\{1, 13, 2, 23\}\\
\midrule
421:1&0,0,\frac{14}{51} \sqrt{3} e^{12},\frac{7}{51} \sqrt{34} e^{13}&\begin{array}{cc}
(\frac{28}{51},-\frac{7}{17},\frac{7}{51},\frac{35}{51})\\
+\frac{7}{51} \sqrt{66} e^1\otimes e_2
\end{array}
&\{13, 234\}\\
\midrule
41:1&0,0,0,\frac{2}{17} \sqrt{22} e^{12}&\begin{array}{cc}
(\frac{15}{17},-\frac{7}{17},\frac{6}{17},\frac{8}{17})\\
+\frac{22}{17} e^1\otimes e_2
\end{array}
&\{134, 14, 234, 24\}\\
\midrule
41:1&0,0,0,\frac{1}{3} \sqrt{6} e^{12}&\begin{array}{cc}
(0,0,1,0)\\
+\frac{2}{3} \sqrt{3} e^3\otimes e_4
\end{array}
&\{123, 14, 3\}\\
\midrule
41:1&0,0,0,\frac{1}{3} \sqrt{6} e^{12}&\begin{array}{cc}
(\frac{2}{3},0,-\frac{1}{3},\frac{2}{3})\\
+\frac{2}{3} \sqrt{3} e^1\otimes e_3
\end{array}
&\{12, 14, 234, 3\}\\
\midrule
41:1&0,0,0,\frac{2}{9} \sqrt{6} e^{12}&\begin{array}{cc}
(\frac{1}{3},-\frac{1}{3},\frac{2}{3},0)\\
+\frac{2}{3} e^1\otimes e_2+\frac{4}{9} \sqrt{3} e^3\otimes e_4
\end{array}
&\{14, 24\}\\
\midrule
4:1&0,0,0,0&\begin{array}{cc}
(1,-\frac{1}{3},\frac{1}{3},\frac{1}{3})\\
+\frac{4}{3} e^1\otimes e_2
\end{array}
&\begin{array}{c}
\{1, 13, 134,\\
2, 23, 234\}
\end{array}
\\
\midrule
4:1&0,0,0,0&\begin{array}{cc}
(\frac{3}{5},\frac{3}{5},-\frac{1}{5},-\frac{1}{5})\\
+\frac{4}{5} e^1\otimes e_3+\frac{4}{5} e^2\otimes e_4
\end{array}
&\{12, 14, 34\}\\
\end{longtable}
}

\begin{footnotesize}
{\setlength{\tabcolsep}{2pt}
\begin{longtable}[htpc]{>{\ttfamily}c C >{\renewcommand{\arraystretch}{1.2}}C C}
\caption{\label{table:5}Solutions of~\eqref{eqn:generalizednilsoliton} obtained with Algorithm~\ref{alg:classification} with $n=5$}\\
\toprule
\textnormal{Name $\Delta$} & \g &   D & \mathbf{S} \\
\midrule
\endfirsthead
\multicolumn{4}{c}{\tablename\ \thetable\ -- \textit{Continued from previous page}} \\
\toprule
\textnormal{Name $\Delta$} & \g &   D & \mathbf{S} \\
\midrule
\endhead
\bottomrule\\[-7pt]
\multicolumn{4}{c}{\tablename\ \thetable\ -- \textit{Continued to next page}} \\
\endfoot
\bottomrule\\[-7pt]
\endlastfoot
5321:1&0,0,\frac{3}{10} \sqrt{5} e^{12},\frac{3}{10} \sqrt{2} e^{13},\frac{3}{10} \sqrt{5} e^{14}&\begin{array}{cc}
(-\frac{3}{10},\frac{3}{4},\frac{9}{20},\frac{3}{20},-\frac{3}{20})\\
+\frac{3}{10} \sqrt{14} e^2\otimes e_5
\end{array}
&\{124, 135\}\\
\midrule
5321:1&\begin{array}{c}
0,0,\frac{11}{159} \sqrt{6} e^{12},\frac{11}{159} \sqrt{106} e^{13},\\
\frac{22}{53} \sqrt{3} e^{14}
\end{array}
&\begin{array}{cc}
(\frac{55}{159},-\frac{22}{53},-\frac{11}{159},\frac{44}{159},\frac{33}{53})\\
+\frac{22}{159} \sqrt{57} e^1\otimes e_2
\end{array}
&\{14, 2\}\\
\midrule
532:1&0,0,\frac{10}{31} e^{12},\frac{5}{31} \sqrt{3} e^{13},\frac{5}{31} \sqrt{3} e^{23}&\begin{array}{cc}
(\frac{15}{31},-\frac{10}{31},\frac{5}{31},\frac{20}{31},-\frac{5}{31})\\
+\frac{25}{31} e^1\otimes e_2+\frac{25}{31} e^4\otimes e_5
\end{array}
&\{135, 234\}\\
\midrule
521:1&0,0,0,\frac{4}{15} \sqrt{10} e^{12},\frac{4}{15} \sqrt{3} e^{14}&\begin{array}{cc}
(-\frac{2}{15},\frac{2}{3},-\frac{2}{5},\frac{8}{15},\frac{2}{5})\\
+\frac{4}{15} \sqrt{21} e^2\otimes e_3
\end{array}
&\{125, 134, 245, 3\}\\
\midrule
521:1&0,0,0,\frac{5}{51} \sqrt{6} e^{12},\frac{5}{51} \sqrt{34} e^{14}&\begin{array}{cc}
(\frac{10}{51},-\frac{5}{17},\frac{10}{17},-\frac{5}{51},\frac{5}{51})\\
+\frac{5}{51} \sqrt{42} e^1\otimes e_2+\frac{5}{51} \sqrt{42} e^3\otimes e_5
\end{array}
&\{134, 245\}\\
\midrule
521:1&0,0,0,\frac{4}{15} \sqrt{3} e^{12},\frac{4}{15} \sqrt{10} e^{14}&\begin{array}{cc}
(-\frac{2}{15},\frac{1}{5},1,\frac{1}{15},-\frac{1}{15})\\
+\frac{4}{15} \sqrt{21} e^3\otimes e_5
\end{array}
&\{125, 134, 245, 3\}\\
\midrule
521:1&0,0,0,\frac{4}{21} \sqrt{14} e^{12},\frac{4}{21} \sqrt{14} e^{14}&\begin{array}{cc}
(\frac{1}{3},-\frac{1}{21},-\frac{3}{7},\frac{2}{7},\frac{13}{21})\\
+\frac{4}{21} \sqrt{30} e^1\otimes e_3
\end{array}
&\{125, 14, 2345, 3\}\\
\midrule
521:1&0,0,0,\frac{4}{15} \sqrt{3} e^{12},\frac{4}{15} \sqrt{10} e^{14}&\begin{array}{cc}
(\frac{17}{30},-\frac{1}{2},\frac{3}{10},\frac{1}{15},\frac{19}{30})\\
+\frac{4}{15} \sqrt{21} e^1\otimes e_2
\end{array}
&\{134, 14, 2345, 245\}\\
\midrule
521:2&
\begin{array}{c}
0,0,0,\frac{6}{17} \sqrt{6} e^{12},\\
\frac{6}{17} e^{13}+\frac{6}{17} \sqrt{2} e^{24}
\end{array}
&\begin{array}{cc}
(\frac{12}{17},-\frac{3}{17},-\frac{6}{17},\frac{9}{17},\frac{6}{17})\\
+\frac{12}{17} \sqrt{3} e^1\otimes e_3
\end{array}
&\{234, 3\}\\
\midrule
521:2&
\begin{array}{c}0,0,0,\frac{36}{115} \sqrt{3} e^{12},\\
\frac{9}{506} \sqrt{110} e^{13}+\frac{9}{2530} \sqrt{2090} e^{24}
\end{array}
&\begin{array}{cc}
(-\frac{9}{23},\frac{36}{115},\frac{72}{115},-\frac{9}{115},\frac{27}{115})\\
+\sqrt{\frac{143127}{290950}} e^2\otimes e_1-\sqrt{\frac{7533}{11638}} e^3\otimes e_4
\end{array}
&\{145, 24\}\\
\midrule
52:1&0,0,0,\frac{1}{19} \sqrt{102} e^{12},\frac{6}{19} e^{13}&\begin{array}{cc}
(\frac{3}{19},\frac{6}{19},-\frac{6}{19},\frac{9}{19},-\frac{3}{19})\\
+\frac{2}{19} \sqrt{33} e^1\otimes e_3+\frac{5}{19} \sqrt{6} e^2\otimes e_5
\end{array}
&\{145, 35\}\\
\midrule
52:1&0,0,0,\frac{3}{11} e^{12},\frac{3}{11} e^{13}&\begin{array}{cc}
(\frac{1}{11},\frac{6}{11},-\frac{3}{11},\frac{7}{11},-\frac{2}{11})\\
+\frac{9}{11} e^2\otimes e_3+\frac{9}{11} e^4\otimes e_5
\end{array}
&\{125, 134, 24, 35\}\\
\midrule
52:1&0,0,0,\frac{1}{6} \sqrt{6} e^{12},\frac{1}{6} \sqrt{6} e^{13}&\begin{array}{cc}
(-\frac{1}{3},\frac{1}{3},\frac{1}{3},0,0)\\
+\frac{1}{6} \sqrt{10} e^2\otimes e_5-\frac{1}{6} \sqrt{10} e^3\otimes e_4
\end{array}
&\{123, 145, 24\}\\
\midrule
52:1&0,0,0,\frac{1}{6} \sqrt{6} e^{12},\frac{1}{6} \sqrt{6} e^{13}&\begin{array}{cc}
(-\frac{1}{3},\frac{1}{3},\frac{1}{3},0,0)\\
+\frac{1}{6} \sqrt{10} e^2\otimes e_5+\frac{1}{6} \sqrt{10} e^3\otimes e_4
\end{array}
&\{123, 145, 24\}\\
\midrule
52:1&0,0,0,\frac{1}{14} \sqrt{91} e^{12},\frac{1}{14} \sqrt{91} e^{13}&\begin{array}{cc}
(-\frac{3}{14},\frac{3}{4},\frac{1}{28},\frac{15}{28},-\frac{5}{28})\\
+\frac{1}{7} \sqrt{65} e^2\otimes e_5
\end{array}
&\{123, 145, 24, 35\}\\
\midrule
52:1&0,0,0,\frac{2}{53} \sqrt{29} e^{12},\frac{6}{53} \sqrt{58} e^{13}&\begin{array}{cc}
(\frac{36}{53},-\frac{22}{53},-\frac{3}{53},\frac{14}{53},\frac{33}{53})\\
+\frac{2}{53} \sqrt{1102} e^1\otimes e_2
\end{array}
&\{13, 15, 2, 235\}\\
\midrule
51:1&0,0,0,0,\frac{1}{3} \sqrt{3} e^{12}&\begin{array}{cc}
(\frac{1}{4},\frac{1}{4},-\frac{1}{4},-\frac{1}{4},\frac{1}{2})\\
+\frac{1}{6} \sqrt{15} e^1\otimes e_3+\frac{1}{6} \sqrt{15} e^2\otimes e_4
\end{array}
&\{12, 145, 34\}\\
\midrule
51:1&0,0,0,0,\frac{1}{5} \sqrt{7} e^{12}&\begin{array}{cc}
(\frac{9}{10},-\frac{1}{2},\frac{3}{10},\frac{3}{10},\frac{2}{5})\\
+\frac{7}{5} e^1\otimes e_2
\end{array}
&\begin{array}{c}
\{1345, 135, 15, 2345\\
235, 25\}\\
\end{array}
\\
\midrule
51:1&0,0,0,0,\frac{1}{5} \sqrt{7} e^{12}&\begin{array}{cc}
(\frac{1}{5},\frac{1}{5},1,-\frac{2}{5},\frac{2}{5})\\
+\frac{7}{5} e^3\otimes e_4
\end{array}
&\begin{array}{c}
\{123, 124, 135, 145\\
3, 4\}\\
\end{array}
\\
\midrule
51:1&0,0,0,0,\frac{1}{21} \sqrt{322} e^{12}&\begin{array}{cc}
(\frac{2}{3},-\frac{1}{21},-\frac{3}{7},\frac{2}{7},\frac{13}{21})\\
+\frac{1}{21} \sqrt{690} e^1\otimes e_3
\end{array}
&\begin{array}{c}
\{12, 124, 145, 15\\
2345, 235, 3, 34\}\\
\end{array}
\\
\midrule
51:1&0,0,0,0,\frac{1}{21} \sqrt{322} e^{12}&\begin{array}{cc}
(-\frac{1}{21},-\frac{1}{21},1,\frac{2}{7},-\frac{2}{21})\\
+\frac{1}{21} \sqrt{690} e^3\otimes e_5
\end{array}
&\begin{array}{c}
\{123, 1234, 145, 15\\
3, 34\}\\
\end{array}
\\
\midrule
51:1&0,0,0,0,\frac{2}{65} \sqrt{322} e^{12}&\begin{array}{cc}
(\frac{21}{65},-\frac{5}{13},\frac{42}{65},\frac{12}{65},-\frac{4}{65})\\
+\frac{2}{65} \sqrt{690} e^3\otimes e_5+\frac{46}{65} e^1\otimes e_2
\end{array}
&\{145, 15, 245, 25\}\\
\midrule
51:1&0,0,0,0,\frac{1}{3} \sqrt{3} e^{12}&\begin{array}{cc}
(\frac{1}{4},-\frac{1}{6},\frac{7}{12},-\frac{1}{4},\frac{1}{12})\\
+\frac{1}{6} \sqrt{15} e^1\otimes e_4+\frac{1}{6} \sqrt{15} e^3\otimes e_5
\end{array}
&\{123, 15, 245, 34\}\\
\midrule
51:1&0,0,0,0,\frac{2}{17} \sqrt{7} e^{12}&\begin{array}{cc}
(\frac{9}{17},-\frac{5}{17},\frac{10}{17},-\frac{4}{17},\frac{4}{17})\\
+\frac{14}{17} e^1\otimes e_2+\frac{14}{17} e^3\otimes e_4
\end{array}
&\{135, 145, 235, 245\}\\
\midrule
51:2&0,0,0,0,\frac{3}{11} e^{12}+\frac{3}{11} e^{34}&\begin{array}{cc}
(\frac{6}{11},-\frac{3}{11},\frac{6}{11},-\frac{3}{11},\frac{3}{11})\\
+\frac{9}{11} e^1\otimes e_2+\frac{9}{11} e^3\otimes e_4
\end{array}
&\{135, 145, 245\}\\
\midrule
51:2&0,0,0,0,\frac{3}{11} e^{12}+\frac{3}{11} e^{34}&\begin{array}{cc}
(\frac{6}{11},-\frac{3}{11},\frac{6}{11},-\frac{3}{11},\frac{3}{11})\\
+\frac{9}{11} e^1\otimes e_2-\frac{9}{11} e^3\otimes e_4
\end{array}
&\{135, 145, 245\}\\
\midrule
51:2&0,0,0,0,\frac{3}{11} e^{12}+\frac{3}{11} e^{34}&\begin{array}{cc}
(\frac{6}{11},-\frac{3}{11},\frac{6}{11},-\frac{3}{11},\frac{3}{11})\\
+\frac{9}{11} e^1\otimes e_4+\frac{9}{11} e^3\otimes e_2
\end{array}
&\{12, 135, 245\}\\
\midrule
51:2&0,0,0,0,\frac{6}{13} e^{12}+\frac{6}{13} e^{34}&\begin{array}{cc}
(\frac{12}{13},-\frac{6}{13},\frac{3}{13},\frac{3}{13},\frac{6}{13})\\
+\frac{18}{13} e^1\otimes e_2
\end{array}
&\{135, 235\}\\
\midrule
5:1&0,0,0,0,0&\begin{array}{cc}
(\frac{7}{12},\frac{7}{12},-\frac{1}{4},-\frac{1}{4},\frac{1}{6})\\
+\frac{5}{6} e^1\otimes e_3+\frac{5}{6} e^2\otimes e_4
\end{array}
&\begin{array}{c}
\{12, 125, 14,\\
145, 34, 345\}
\end{array}
\\
\midrule
5:1&0,0,0,0,0&\begin{array}{cc}
(1,-\frac{3}{7},\frac{2}{7},\frac{2}{7},\frac{2}{7})\\
+\frac{10}{7} e^1\otimes e_2
\end{array}
&\begin{array}{c}
\{1, 13, 134, 1345,\\
2, 23, 234, 2345\}
\end{array}
\\
\end{longtable}
}
\end{footnotesize}

\FloatBarrier

In order to show that the construction of this paper yields new metrics, we compare the Einstein solvmanifolds we obtain to the known Einstein metrics of pseudo-Iwasawa type. Specifically, we compare with the Einstein solvmanifolds obtained by extending a nilsoliton of dimension $\leq4$ that admits an orthonormal nice basis.
\begin{proposition}
\label{prop:notthesame}
The Einstein solvmanifolds obtained from the metrics of Table~\ref{table:34} are not isometric to any pseudo-Iwasawa Einstein solvmanifolds obtained by extending a nilsoliton admitting a nice orthonormal basis.
\end{proposition}
\begin{proof}
Diagonal nilsoliton metrics on irreducible nice Lie algebras of dimension $3$ and $4$ are classified in~\cite{Conti_Rossi_2022}. Using the same methods, the classification can be extended to include the reducible nice Lie algebra $(0,0,0,e^{12})$. For each nilsoliton obtained in this way, we can write down explicitly the resulting Einstein solvmanifold $\tilde\g$ and compute the curvature tensor. The metric is determined only up to a multiple; we will fix a normalization, so the statement must be proved up to isometry and rescaling.

By raising an index, we view the curvature tensor as an element of $\Lambda^2\tilde\g^*\otimes\tilde\g\otimes\tilde\g$ rather than $\Lambda^2\tilde\g^*\otimes\tilde\g^*\otimes\tilde\g$, obtaining an endomorphism $R\colon\Lambda^2\tilde\g\to\Lambda^2\tilde\g$. We use the conjugacy class of $R$ as an invariant. More precisely, we determine the characteristic polynomial and whether $R$ is diagonalizable; notice that $R$ is symmetric relative to the scalar product induced by $\tilde g$ on $\Lambda^2\tilde\g$, but the latter is not definite, so the spectral theorem does not apply. It turns out that the trace of $R$ is nonzero in each case; in order to account for rescalings, we consider the characteristic polynomial of the
normalized operator $\tilde{R} =\frac{1}{\Tr R} R$.

We illustrate the computation comparing the metrics obtained by extending the Heisenberg Lie algebra. Extending the diagonal nilsoliton metric, we obtain the Lie algebra $(\frac{1}{6} \sqrt{3} e^{14},\frac{1}{6} \sqrt{3} e^{24},\frac{1}{3} \sqrt{3} e^{12}+\frac{1}{3} \sqrt{3} e^{34},0)$ with metric $e^{1}\otimes e^1-e^2\otimes e^2-e^3\otimes e^3+e^4\otimes e^4$; the Riemann operator $R\colon\Lambda^2\to\Lambda^2$ in the basis $\{e^{12},e^{13},e^{14},e^{23},e^{24},e^{34}\}$ is given by the matrix
\[
\begin{pmatrix}
 \frac{1}{3} & 0 & 0 & 0 & 0 & \frac{1}{6} \\
 0 & \frac{1}{12} & 0 & 0 & \frac{1}{12} & 0 \\
 0 & 0 & \frac{1}{12} & \frac{1}{12} & 0 & 0 \\
 0 & 0 & \frac{1}{12} & \frac{1}{12} & 0 & 0 \\
 0 & \frac{1}{12} & 0 & 0 & \frac{1}{12} & 0 \\
 \frac{1}{6} & 0 & 0 & 0 & 0 & \frac{1}{3} \\
\end{pmatrix},\]
which is diagonalizable.
On the other hand, the extension of the Heisenberg Lie algebra corresponding to the
first entry of Table~\ref{table:34} yields the Einstein Lie algebra $(\frac{6}{7} e^{14},\frac{8}{7} e^{14}-\frac{2}{7} e^{24},\frac{4}{7} e^{12}+\frac{4}{7} e^{34})$ with metric $e^{1}\otimes e^1-e^2\otimes e^2-e^3\otimes e^3+e^4\otimes e^4$; the Riemann operator $R$ is then given by the matrix
\[
\begin{pmatrix}
 \frac{16}{49} & 0 & 0 & 0 & 0 & \frac{8}{49} \\
 0 & \frac{20}{49} & \frac{8}{49} & -\frac{16}{49} &
   -\frac{4}{49} & 0 \\
 0 & -\frac{8}{49} & -\frac{12}{49} & \frac{12}{49} &
   \frac{16}{49} & 0 \\
 0 & \frac{16}{49} & \frac{12}{49} & -\frac{12}{49} &
   -\frac{8}{49} & 0 \\
 0 & -\frac{4}{49} & -\frac{16}{49} & \frac{8}{49} &
   \frac{20}{49} & 0 \\
 \frac{8}{49} & 0 & 0 & 0 & 0 & \frac{16}{49} \\
\end{pmatrix},\]
which is not diagonalizable. Thus, the metrics are not isometric. Notice that the characteristic polynomial  is not sufficient to distinguish these two particular metrics, since in both cases one obtains
$\det(\lambda \id-\tilde R)=\lambda ^6-\lambda ^5+\frac{\lambda ^4}{3}-\frac{5 \lambda ^3}{108}+\frac{\lambda ^2}{432}$.

To compare the rest of the metrics, we use both diagonalizability and the characteristic polynomial of $\tilde R$, which due to the normalization takes the form $\lambda^N-\lambda^{N-1}+a_2\lambda^{N-2}+\dots$. It turns out that the coefficient $a_2$ is sufficient to distinguish metrics obtained from diagonal nilsolitons from those obtained by nondiagonal triples.

In Table~\ref{tbl:diagonalnilsolitonsa2}, we list Einstein solvmanifolds obtained by extending a diagonal nice nilsoliton, the signature and the corresponding value of $a_2$; a check mark~$\checkmark$ in the last column indicates that $R$ is diagonalizable over $\C$. We do not include positive-definite metrics, since the purpose is a comparison with the metrics obtained from Table~\ref{table:34}, which are indefinite by design. Table~\ref{tbl:gennilsolitonsa2} contains analogous data, starting with the metrics of Table~\ref{table:34}.

In both tables, only one entry is given up to equivalence in the sense of Lemma~\ref{lemma:equivalent}. Notice that when more than one signature arises, the signatures are related by an element of $\ker M_{\Delta,2}$. The corresponding metrics are then related by a so-called Wick rotation, so it is not surprising that the Riemann tensor is the same (\cite{Helleland,Conti_Rossi_2022}). Metrics related in this way appear in the same row in the table.

Since $a_2$ is an invariant up to isometry and rescaling, the statement follows by comparing the rows of the two tables.

\begin{table}[thp]
{\renewcommand{\arraystretch}{1.2}\centering
\caption{\label{tbl:diagonalnilsolitonsa2} Indefinite Einstein solvmanifolds obtained by extending a diagonal nice nilsoliton}
\hspace*{-1cm}
\begin{tabular}{C C C C}
\toprule
\g& \mathbf{S}&a_2 &\textnormal{diag.?}\\
\midrule
\frac{1}{6} \sqrt{3} e^{14},\frac{1}{6} \sqrt{3} e^{24},\frac{1}{3} \sqrt{3} e^{12}+\frac{1}{3} \sqrt{3} e^{34},0&\{12,23\}&
\frac{1}{3}&
\checkmark
\\
\frac{1}{6} \sqrt{6} e^{14},\frac{1}{6} \sqrt{6} e^{24},\frac{1}{6} \sqrt{6} e^{34},0&\{1,12,123\}&
\frac{5}{12}&
\checkmark
\\
\midrule
\multicolumn{1}{L}{\frac{1}{30} \sqrt{15} e^{15},\frac{1}{15} \sqrt{15} e^{25},\frac{1}{3} \sqrt{3} e^{12}+\frac{1}{10} \sqrt{15} e^{35},}&
\multirow{2}{*}{$\{124,13,234\}$}&
\multirow{2}{*}{$\frac{2099}{5625}$}&
\multirow{2}{*}{$\checkmark$}\\
\multicolumn{1}{R}{\frac{1}{3} \sqrt{3} e^{13}+\frac{2}{15} \sqrt{15} e^{45},0}&&&
\\
\frac{1}{33} \sqrt{66} e^{15},\frac{1}{33} \sqrt{66} e^{25},\frac{1}{22} \sqrt{66} e^{35},\frac{1}{3} \sqrt{3} e^{12}+\frac{2}{33} \sqrt{66} e^{45},0&\{12,123,134,14,3\}&
\frac{3683}{9075}&
\checkmark
\\
\frac{1}{4} \sqrt{2} e^{15},\frac{1}{4} \sqrt{2} e^{25},\frac{1}{4} \sqrt{2} e^{35},\frac{1}{4} \sqrt{2} e^{45},0&\{1,12,123,1234\}&
\frac{9}{20}&
\checkmark\\
\bottomrule
\end{tabular}
}
\end{table}

\begin{table}[thp]
{\renewcommand{\arraystretch}{1.2}
\centering
\caption{\label{tbl:gennilsolitonsa2} Einstein solvmanifolds obtained by extending a nondiagonal triple}
\hspace*{-0.7cm}
\begin{tabular}{C C C C}
\toprule
\g& \mathbf{S}&a_2 &\textnormal{diag.?}\\
\midrule
\frac{6}{7} e^{14},\frac{8}{7} e^{14}-\frac{2}{7} e^{24},\frac{4}{7} e^{12}+\frac{4}{7} e^{34},0&\{13,23\}&
\frac{1}{3}&\\
e^{14},\frac{6}{5} e^{14}-\frac{1}{5} e^{24},\frac{2}{5} e^{34},0&\{1,13,2,23\}&
\frac{5}{12}&\\
\midrule
\multicolumn{1}{L}{\frac{28}{51} e^{15},\frac{7}{51} \sqrt{66} e^{15}-\frac{7}{17} e^{25},\frac{14}{51} \sqrt{3} e^{12}+\frac{7}{51} e^{35},}&
\multirow{2}{*}{$\{13,234\}$}&\multirow{2}{*}{$\frac{113}{2700}$}&\multirow{2}{*}{$\checkmark$}\\
\multicolumn{1}{R}{\frac{7}{51} \sqrt{34} e^{13}+\frac{35}{51} e^{45},0}&&&
\\
\frac{15}{17} e^{15},\frac{22}{17} e^{15}-\frac{7}{17} e^{25},\frac{6}{17} e^{35},\frac{2}{17} \sqrt{22} e^{12}+\frac{8}{17} e^{45},0&\{134,14,234,24\}&
\frac{3683}{9075}&
\\
0,0,e^{35},\frac{1}{3} \sqrt{6} e^{12}+\frac{2}{3} \sqrt{3} e^{35},0&\{123,14,3\}&
\frac{3}{5}&
\checkmark
\\
\frac{2}{3} e^{15},0,\frac{2}{3} \sqrt{3} e^{15}-\frac{1}{3} e^{35},\frac{1}{3} \sqrt{6} e^{12}+\frac{2}{3} e^{45},0&\{12,14,234,3\}&
\frac{1}{15}&
\checkmark
\\
\frac{1}{3} e^{15},\frac{2}{3} e^{15}-\frac{1}{3} e^{25},\frac{2}{3} e^{35},\frac{2}{9} \sqrt{6} e^{12}+\frac{4}{9} \sqrt{3} e^{35},0&\{14,24\}&
\frac{3}{5}&
\\
e^{15},\frac{4}{3} e^{15}-\frac{1}{3} e^{25},\frac{1}{3} e^{35},\frac{1}{3} e^{45},0&\{1,13,134,2,23,234\}&
\frac{9}{20}&
\\
\frac{3}{5} e^{15},\frac{3}{5} e^{25},\frac{4}{5} e^{15}-\frac{1}{5} e^{35},\frac{4}{5} e^{25}-\frac{1}{5} e^{45},0&\{12,14,34\}&
\frac{9}{20}&
\\
\bottomrule
\end{tabular}
}
\end{table}
\end{proof}

\begin{remark}
In the Riemannian case, an Einstein solvmanifold is determined by the nilradical: two Einstein solvmanifolds with isomorphic nilradicals are isometric (see \cite{Heber_1998,Lauret_2001}). In the indefinite case this is not true: indeed, there exist nilpotent Lie algebras with two nonisometric nilsoliton metrics, and this implies that the corresponding Einstein standard extensions are nonisometric (see \cite[Remark~2.6]{Conti_Rossi_best_before}).

The construction of this paper shows in addition that even if one fixes the metric on the nilradical, the standard Einstein extension is not unique. For instance, consider the two extensions of the Heisenberg Lie algebra given explicitly in the proof of Proposition~\ref{prop:notthesame}. The metrics induced on the nilradical are diagonal metrics on the Heisenberg Lie algebra, so they coincide up to a change of basis and rescaling. However, the extensions are not isometric.

Moreover, we easily see that the pseudo-Iwasawa extension is isomorphic to $M^{12}$ of~\cite{deGraaf:SolvableLie}, whilst the other  is isomorphic to $M^{13}_{\frac{3}{4}}$: thus, they are neither isometric nor isomorphic.
\end{remark}

\FloatBarrier

\medskip
\small\noindent D. Conti: Dipartimento di Matematica, Università di Pisa, largo Bruno Pontecorvo 6, 56127 Pisa, Italy.\\
\texttt{diego.conti@unipi.it}\\
\small\noindent F. A. Rossi: Dipartimento di Matematica e Informatica, Universit\`a degli studi di Perugia, via Vanvitelli 1, 06123 Perugia, Italy.\\
\texttt{federicoalberto.rossi@unipg.it}
\small\noindent R.~Segnan Dalmasso: Dipartimento di Matematica e Applicazioni, Universit\`a di Milano Bicocca, via Cozzi 55, 20125 Milano, Italy.\\
\texttt{romeo.segnandalmasso@unimib.it}\\

\printbibliography

\end{document}